\documentclass[hidelinks]{amsart}
\usepackage{a4wide}
\usepackage{amsmath}
\usepackage{stackrel}
\usepackage{amssymb}
\usepackage{graphicx}
\usepackage[numbers]{natbib}
\usepackage[english]{babel}
\usepackage[utf8]{inputenc}
\usepackage[T1]{fontenc}
\usepackage{enumitem}
\usepackage{multicol}
\usepackage{mathtools}
\usepackage{multirow}
\usepackage{array}
\usepackage{booktabs}
\usepackage{tikz}
\usetikzlibrary{arrows}
\usetikzlibrary{shapes}
\usetikzlibrary{positioning}
\usetikzlibrary{cd}
\usetikzlibrary{babel}
\usepackage[toc,page]{appendix}
\usepackage[all]{xy}
\usepackage{float}
\usepackage{centernot}
\usepackage{srcltx}
\usepackage{hyperref}
\usepackage[capitalise]{cleveref}

\newtheorem{theorem}{Theorem}
\newtheorem{lemma}[theorem]{Lemma}

\theoremstyle{definition}

\newtheorem{remark}[theorem]{Remark}
\newtheorem*{theorem*}{Theorem}
\newtheorem*{lemma*}{Lemma}
\newtheorem*{proposition*}{Proposition}
\newtheorem*{corollary*}{Corollary}
\newtheorem*{definition*}{Definition}
\newtheorem*{remark*}{Remark}

\makeatletter
\renewcommand{\@biblabel}[1]{[#1]\hfill}
\makeatother

\setlist[itemize]{label=$\diamond$, leftmargin=0cm, itemindent=0.8cm}
\setlist[enumerate]{leftmargin=0cm, itemindent=1cm}

\pdfoutput=1

\begin{document}

\title{On the additivity of the little cubes operads}

\author{Miguel Barata}
\address{Mathematisch Instituut,
         Universiteit Utrecht,
         Postbus 80010,
         Utrecht, 3508 TA, The Netherlands}
\email{m.lourencohenriquesbarata@uu.nl}

\author{Ieke Moerdijk}
\address{Mathematisch Instituut,
         Universiteit Utrecht,
         Postbus 80010,
         Utrecht, 3508 TA, The Netherlands}
\email{i.moerdijk@uu.nl}

\begin{abstract}

   We give a new proof of Dunn's additivity for the little $n$-cubes operads $C_n$, which has the advantage of being considerably shorter and simpler than the proofs in \cite{Dunn, Brinkmeier}. At the end we remark on how our proof can be adjusted to the case of the tensor product of a finite number of factors.
\end{abstract}

\maketitle

\section{Introduction}

The purpose of this note is to provide a simpler proof of the additivity property of the little $n$-cubes operads $C_n$, usually known as \textit{Dunn additivity}, precisely stated in Theorem \ref{add} below. Before stating it, recall that a morphism $f\colon \mathbf{P} \to \mathbf{Q}$ of (uncoloured) topological operads is a \textit{weak equivalence} if the map on operations $f\colon \mathbf{P}(j) \to \mathbf{Q}(j)$ is a weak homotopy equivalence for all $j \geq 0$ (Dunn called these local $\Sigma$-equivalences). We will write $\mathbf{P} \simeq \mathbf{Q}$ if there exists such a  weak equivalence, or a zig-zag of such, between $\mathbf{P}$ and $\mathbf{Q}$. 

\begin{theorem}
$C_n \otimes C_m \simeq C_{n+m}$ for $n, m \geq 1$, where $\otimes$ denotes the Boardman--Vogt tensor product of topological operads.
\label{add}
\end{theorem}

The first version of this statement was proven by Dunn in \cite{Dunn}, where he showed that $C_1^{\otimes n} \simeq C_n$. However, it is important to remark that Theorem \ref{add} is not an immediate consequence of Dunn's result since in general the Boardman--Vogt tensor product need not preserve weak equivalences as is well known: for instance, for the associative operad $\mathbf{Ass}$ we have that $\mathbf{Ass} \otimes \mathbf{Ass}$ is the commutative operad $\mathbf{Comm}$; however $C_1 \simeq \mathbf{Ass}$ and $C_1 \otimes C_1 \simeq C_2$ by Dunn's additivity, but $C_2 \not\simeq \mathbf{Comm}$. A proof of Theorem \ref{add} is given in \cite{Brinkmeier} following the general idea of Dunn's proof (see also \cref{history} for some comments on this proof). In this text we aim to give a different proof of the additivity statement, in which we rely on the description of the operations of the operad $C_n \otimes C_m$ as being generated, under composition, by elementary operations of the form $p \otimes 1$ and $1 \otimes q$, where $p \in C_n(a), q \in C_m(b)$. In Remark \ref{multi} we comment on how our methods can be generalized to the case of a finite tensor product $C_{k_1} \otimes \cdots \otimes C_{k_n}$.

\begin{remark}
    After this paper appeared as a preprint on arXiv with number 2205.12875, our methods have been used by Ben Szczesny in \citep[Ch. 4]{Szczesny} to prove Dunn additivity in an equivariant setting, and by work (to appear) of Victor Carmona and Anja Scraka to prove a version of additivity for the Swiss-cheese operad.
    \label{other_work}
\end{remark}

\noindent \textbf{Acknowledgements.} We are grateful to the anonymous referee for useful comments on the paper. The first author would like to thank Victor Carmona for talking about his work and mentioning the PhD thesis \cite{Szczesny} that we referred to in \cref{other_work}. The first author was supported by the Dutch Research Council (NWO) through the grant 613.009.147.

\section{The proof of Theorem 1}

To begin with, let us describe how one can construct an operad map $\mu\colon C_n \otimes C_m \to C_{n+m}$. Consider the morphisms
$$\mu_1\colon C_n \rightarrow C_{n+m} \hspace{1em} \mathrm{and} \hspace{1em} \mu_2\colon C_m \rightarrow C_{n+m}$$
induced by the inclusions (of spaces) of the first $n$ coordinates $\mathbb{R}^n \to \mathbb{R}^{n+m}$ and of the last $m$ coordinates $\mathbb{R}^m \to \mathbb{R}^{n+m}$ respectively. These are easily seen to give rise to an operad map on the Boardman--Vogt tensor product $C_n \otimes C_m$,
$$ \mu\colon C_n \otimes C_m \longrightarrow C_{n+m},$$
whose image is a suboperad $C_n \vert C_m \subseteq C_{n+m}$ which is defined in \citep[Sec. 8]{Brinkmeier}. It follows that $\mu$ factors through the operad inclusion $C_n \vert C_m \subseteq C_{n+m}$ as an operad morphism $\mu'\colon C_n \otimes C_m \to C_n \vert C_m$ which is surjective on operations.

The next step in both of the proofs by Dunn and Brinkmeier is to compare $C_n \vert C_m$ with the operad of decomposable little $(n+m)$-cubes $D_{n+m} \subseteq C_{n+m}$, see \citep[Def. 2.1.]{Dunn} for a definition of $D_k$ and \cref{history} below for a small correction to this definition. They begin by observing the following concerning $D_k$:
\begin{enumerate}
    \item[(i)] The inclusion $D_k \hookrightarrow C_k$ is a  weak equivalence.
    \item[(ii)] Letting $D_n \vert D_m \subseteq C_{n+m}$ denote the operad given by the image of $\mu$ when restricted to $D_n \otimes D_m$, the inclusion $D_n \vert D_m \hookrightarrow C_n \vert C_m$ is a  weak equivalence.
    \item[(iii)] The operad $D_n \vert D_m$ coincides with $D_{n+m}$.
\end{enumerate}

We'll not repeat the proofs of these observations as these are rather elementary: the required homotopies are simply obtained by radially contracting each of the cubes $c \in C_{n+m}(j)$ to their centers and observing that, after some moment in this contraction, $c$ will be an element of $D_{n+m}(j)$. Putting these results together, it follows after examining the commutative diagram
\begin{equation}
\begin{tikzcd}
C_n \otimes C_m \arrow[rr, "\mu"] \arrow[d, "\mu'"'] &                                                                    & C_{n+m}                            \\
C_n|C_m                                              & D_n|D_m \arrow[l, "\simeq"', hook'] \arrow[r, -, double equal sign distance] & D_{n+m} \arrow[u, "\simeq"', hook]
\end{tikzcd}
\label{main_diag}
\end{equation}
that to prove Theorem \ref{add} it is sufficient to show that $\mu'\colon C_n \otimes C_m \to C_n \vert C_m$ is an isomorphism of topological operads. This is the difficult step in both Dunn's and Brinkmeier's proofs and so our purpose is to give a shorter and simpler explanation of why this map is indeed an isomorphism.

We start with the following lemma, which will allow us to split up the inductive step of the proof of the injective statement in Lemma \ref{homeo} into more tractable cases.

\begin{lemma}
Let $n,m \in \mathbb{N}$ and consider an operation $c $ in $C_n \otimes C_m$ of arity $j > 1$. Then $c$ can be written as
\begin{equation}
 c = (p \otimes 1) \circ (t_1, \ldots, t_a) \hspace{1em} or \hspace{1em} c=(1 \otimes q) \circ (s_1, \ldots, s_b),
 \label{norm}
 \end{equation}
where $p \in C_n(a), q \in C_m(b)$ with $a,b > 1$, and $t_i, s_j$ are operations in $C_n \otimes C_m$ of positive arity.
\label{gen_pos}
\end{lemma}
\begin{proof}
Since the operations in $C_n \otimes C_m$ are generated under composition by elements of the form $p \otimes 1$ and $1 \otimes q$, we can certainly write $c$ in one of the forms shown in \eqref{norm}. Thus, it remains to show that our constraints on the arities of $p, q, t_i$ and $s_j$ can be realized. In what follows we will write $\ast$ for the nullary operation of the operads $C_n, C_m$ and $C_n \otimes C_m$; it should always be clear from the context in which operad we are considering $\ast$ so no confusion will arise from this notation.

If $c = (p \otimes 1) \circ (t_1, \ldots, t_a)$ but $t_i = \ast$ holds for some index $i$, then letting $p'$ be the operation $p \circ_i \ast \in C_n(a-1)$ we see that
$$ c= (p' \otimes 1) \circ \left( t_1, \ldots, \hat{t}_i, \ldots, t_a \right)$$
where $\hat{t}_i$ denotes the omission of the $i^{th}$ input $t_i$. Hence, we can assume that none of the $t_i$'s is the nullary operation. Moreover, a simple adaptation of this argument also shows that we can assume that $s_i \neq \ast$ in expression \eqref{norm}.

As for the restriction $a> 1$, we induct on the \textit{length} of $c$, denoted $\ell(c)$, which we define to be the infimum of the number of elementary operations of the the form $p \otimes 1$ and $1 \otimes q$ needed to write $c$ using the composition of such operations in the operad. If $\ell(c)=1$ then $c$ is either $p \otimes 1$ or $1 \otimes q$, where $p$ and $q$ have arity $j > 1$, and so we have nothing to prove. For the induction step, the length of $c$ must be realized by an expression of the form $c = (p \otimes 1) \circ t_1$ or $c=(1 \otimes q) \circ s_1$, where $p, q$ have arity 1 (otherwise we have nothing to prove) and $\ell(t_1), \ell(s_1) < \ell(c)$. If the former holds then, by our induction hypothesis, we know that $t_1$ must be of the form
$$ t_1 = (p' \otimes 1) \circ (t_1', \ldots, t_k') \hspace{1em} \mathrm{or} \hspace{1em} t_1 = (1 \otimes q') \circ (s'_1, \ldots, s'_{\ell}),$$
where $p', q'$ have at least arity 2. For the leftmost case  we can use the relation $(p \otimes 1) \circ (p' \otimes 1) = (p \circ p') \otimes 1$ to conclude the intended; for the other case we use the interchange relation of the tensor product instead. The argument using the expression $c=(1 \otimes q) \circ s_1$ is completely parallel to this one. This finishes our induction.
\end{proof}

\begin{lemma}
Let $j \geq 1$. Then $\mu'\colon (C_n \otimes C_m)(j) \to (C_n \vert C_m)(j)$ is a continuous bijection.
\label{homeo}
\end{lemma}
\begin{proof}
Since $C_n \vert C_m$ is just the image of $\mu$ and thus $\mu'$ is surjective at the level of operations, we only have to check that $\mu'$ is also injective.

For $j=1$ notice that due to the interchange relation, any element of $(C_n \otimes C_m)(1)$ is of the form $(p \otimes 1) \circ (1 \otimes q)$ where $p \in C_n(1)$ and $q \in C_m(1)$. Moreover, it is clear that $\mu'( (p \otimes 1) \circ (1 \otimes q)) = \mu'( (p' \otimes 1) \circ (1 \otimes q'))$ can happen if and only if $p=p', q=q'$.

Suppose that $c_1, c_2 \in (C_n \otimes C_m)(j)$ satisfy $\mu'(c_1)=\mu'(c_2)$. By Lemma \ref{gen_pos}, it is enough to consider the following three scenarios:
\begin{enumerate}
    \item[(i)] We can write
    $$ c_1 = (p \otimes 1) \circ (t_1, \ldots, t_a) \hspace{1em} \mathrm{and} \hspace{1em} c_2=(\overline{p} \otimes 1) \circ (\overline{t_1}, \ldots, \overline{t_{\overline{a}}} ), $$
    where $p \in C_n(a), \overline{p} \in C_n(\overline{a})$ with $a, \overline{a} > 1$ and $t_i, \overline{t_j} \neq \ast$.
    \item[(ii)] We can write
    $$ c_1 = (1 \otimes q) \circ (s_1, \ldots, s_b) \hspace{1em} \mathrm{and} \hspace{1em} c_2=(1 \otimes \overline{q}) \circ (\overline{s_1}, \ldots, \overline{s_{\overline{b}}} ), $$
    where $q \in C_m(b), \overline{q} \in C_m(\overline{b})$ with $b, \overline{b} > 1$ and $s_i, \overline{s_j} \neq \ast$.
    \item[(iii)] We can write
    $$ c_1 = (p \otimes 1) \circ (t_1, \ldots, t_a) \hspace{1em} \mathrm{and} \hspace{1em} c_2=(1 \otimes q) \circ (s_1, \ldots, s_b ), $$
    where $p \in C_n(a), q \in C_m(b)$ with $a, b > 1$ and $t_i, s_j \neq \ast$.
\end{enumerate} 

Proceeding inductively on $j$, suppose first that (i) holds. If $p = \overline{p}$ then, by considering the equality
$$\mu'(p \otimes 1) \circ (\mu'(t_1), \ldots, \mu'(t_a)) = \mu'(p \otimes 1) \circ (\mu'(\overline{t}_1), \ldots, \mu'(\overline{t}_a))$$
together with the injectivity of the composition map $\mu'(p\otimes 1) \circ (-)$ in $C_{n+m}$, we conclude that $\mu'(t_i) = \mu'(\overline{t}_i)$ for all $i=1, \ldots, a$. As $t_i$ and $\overline{t}_i$ have arity strictly smaller than that of $c$ by our hypotheses in (i), the induction hypothesis gives us that $t_i = \overline{t}_i$ and thus $c_1 = c_2$. We remark that this argument also holds for scenario (ii) whenever $q=\overline{q}$.

If $p \neq \overline{p}$ then we claim that there exists an operation $p \wedge \overline{p}$ in $C_n$ satisfying the following two constraints:
\begin{enumerate}
    \item[(a)] There exist operations $p_1, \ldots, p_k$ and $\overline{p}_1, \ldots, \overline{p}_{\overline{k}}$ in $C_n$ such that
     $$
 p \wedge \overline{p}= p \circ (p_1, \ldots, p_k) = \overline{p}\circ (\overline{p}_1, \ldots, \overline{p}_{\overline{k}}).    
$$
    \item[(b)] There exists an operation $c_0=((p \wedge \overline{p}) \otimes 1) \circ (u_1, \ldots, u_\ell)$ in $C_n \otimes C_m$ satisfying $\mu'(c_0) = \mu'(c_1)= \mu'(c_2)$.
\end{enumerate}

To construct this, consider the configurations of cubes determined by $p$ and $\overline{p}$ in $]0,1[^n$. These cubes must intersect each other (otherwise it would contradict $\mu'(c_1)=\mu'(c_2)$) and so let $\alpha_1, \ldots, \alpha_k$ be an enumeration of the cubes resulting from these intersections. From this, we further discard all $\alpha_i$ such that $\mu'(\alpha_i \otimes 1) \cap \mu'(c_1) = \emptyset$. We set $p \wedge \overline{p}$ to be the configuration determined by these remaining cubes, arbitrarily enumerated. Our initial choice of $\alpha_i$'s ensures that (a) is satisfied, and the removal of the extra cubes makes (b) hold.

The figure below pictorially exmplifies how to construct $p \wedge \overline{p}$ when $n=2$: the red squares are elements of $p$, the blue ones represent $\overline{p}$ and the green cubes are the squares associated to the $t_i$'s and the $\overline{t_j}$'s. In our construction, the $\alpha_i$ are the yellow squares, but the only ones we include in $p \wedge \overline{p}$ are the ones containing green rectangles in their interior.

\vspace{-2em}

\begin{figure}[H]
\centerline{\includegraphics[scale=.300]{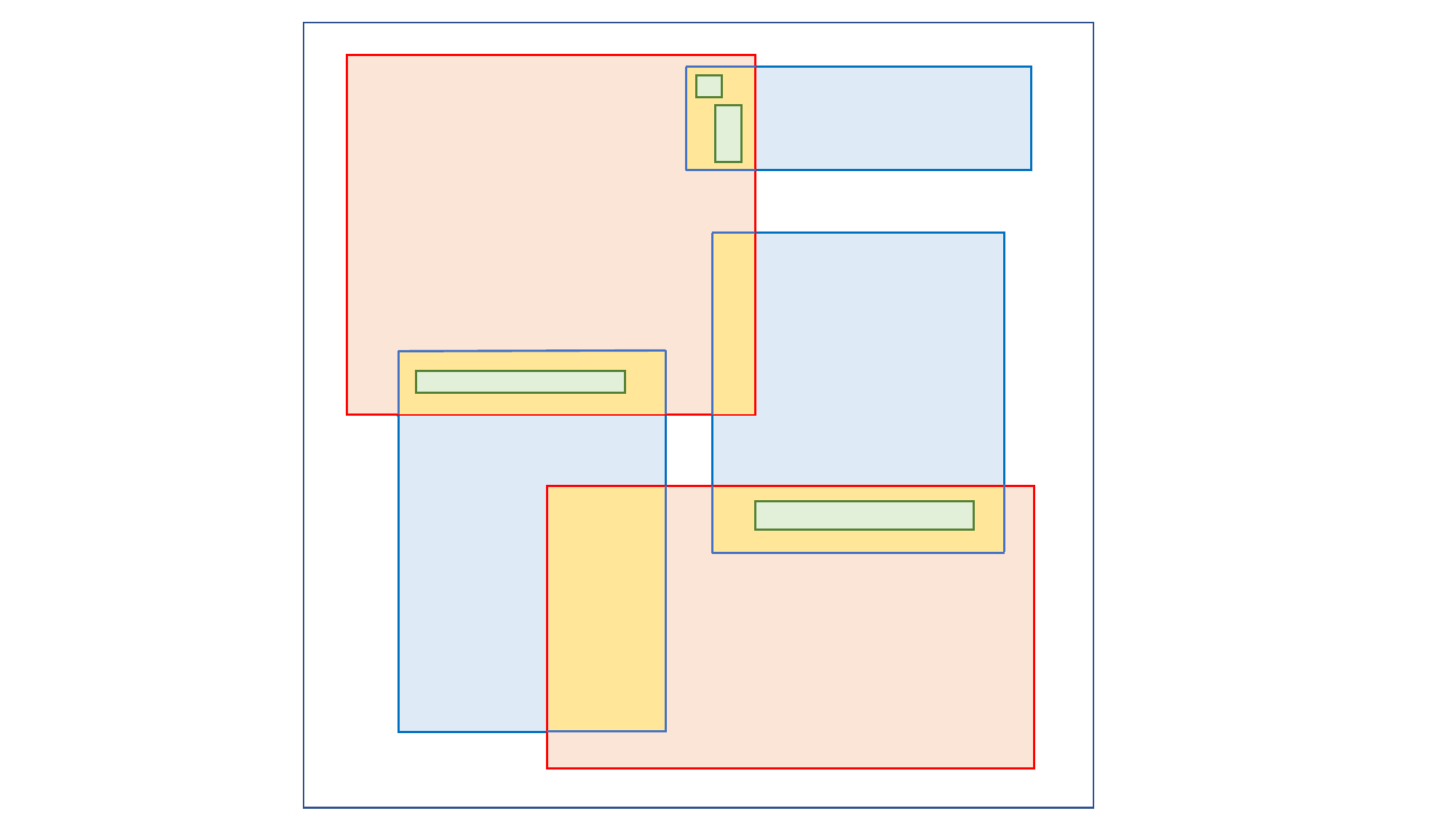}}
\end{figure}

\vspace{-1em}

Starting from the expression for $c_0$ in (b), condition (a) allows us to rewrite $c_0$ as below:
\begin{equation}
c_0 = \big(  (p \otimes 1) \circ (p_1 \otimes 1, \ldots, p_k \otimes 1) \big) \circ (u_1, \ldots, u_{\ell}) = (p \otimes 1) \circ (u'_1, \ldots, u'_a),
\label{adj}
\end{equation} 
and we claim that none of the $u_i'$ can be the nullary operation $\ast$. Indeed, suppose  $u'_i=\ast$ for some index $i$ and write $\beta_i$ for $i^{th}$ cube of $p$. Then $\mu'(\beta_i \otimes 1)$ doesn't intersect $\mu'(c_1)$ in $]0,1[^{n+m}$, which contradicts our initial assumption (i) that $t_i \neq \ast$.

Thus we can apply the special case of (i) which we discussed first to 
$$c_0 = (p \otimes 1) \circ (u'_1, \ldots, u'_a) \hspace{1em} \mathrm{and} \hspace{1em} c_1 = (p \otimes 1) \circ (t_1, \ldots, t_a),$$
 and conclude that $c_0=c_1$. If we instead argued using $\overline{p}$ and wrote $c_0$ as an expression of the form $(\overline{p} \otimes 1) \circ (v_1', \ldots, v_b')$ via property (a), then we would obtain $c_0=c_2$. Piecing all of this together, we have shown that $c_1=c_2$, finishing the proof for (i). The argument for case (ii) follows by symmetry.

If (iii) holds then we claim that each $t_i$ must be of the form $(1 \otimes q) \circ (v_1^i, \ldots, v_b^i)$. Indeed, using our initial condition $\mu'(c_1) = \mu'(c_2)$, we can find an element of this form - call it $t_i'$ for the time being - such that  $\mu'(t'_i) = \mu'(t_i)$ in $C_n \vert C_m$. Since the arity of $t_i$ is strictly less than that of $c_1$ (here we use that the $t_i$'s are not the nullary operation), we can apply the induction hypothesis to conclude that $t_i=t'_i$. 

The figure below exemplifies, for $n=m=1$, how one can find $t_1'$. Here $\mu'(c_1)$ is made up of the green rectangles. The configuration $p \otimes 1$ gives rise to the vertical strips $\mu'(p_i \otimes 1)$ in blue, and the red horizontal ones $\mu'(1 \otimes q_j)$ are defined via the operation $1 \otimes q$. Moreover, $\mu'(c_1)$ must be contained in the intersection of these rectangles, which are the yellow cubes. The element $\mu'(t_1)$, which corresponds, up to rescalling on the horizontal direction, to the green cubes inside $\mu'(p_1 \otimes 1)$, has a subdivision by the horizontal strips, which can then be used to construct our desired $t_1' = (1 \otimes q) \circ (v_1', v_2', v_3')$.

\begin{figure}[H]
\centerline{\includegraphics[scale=.300]{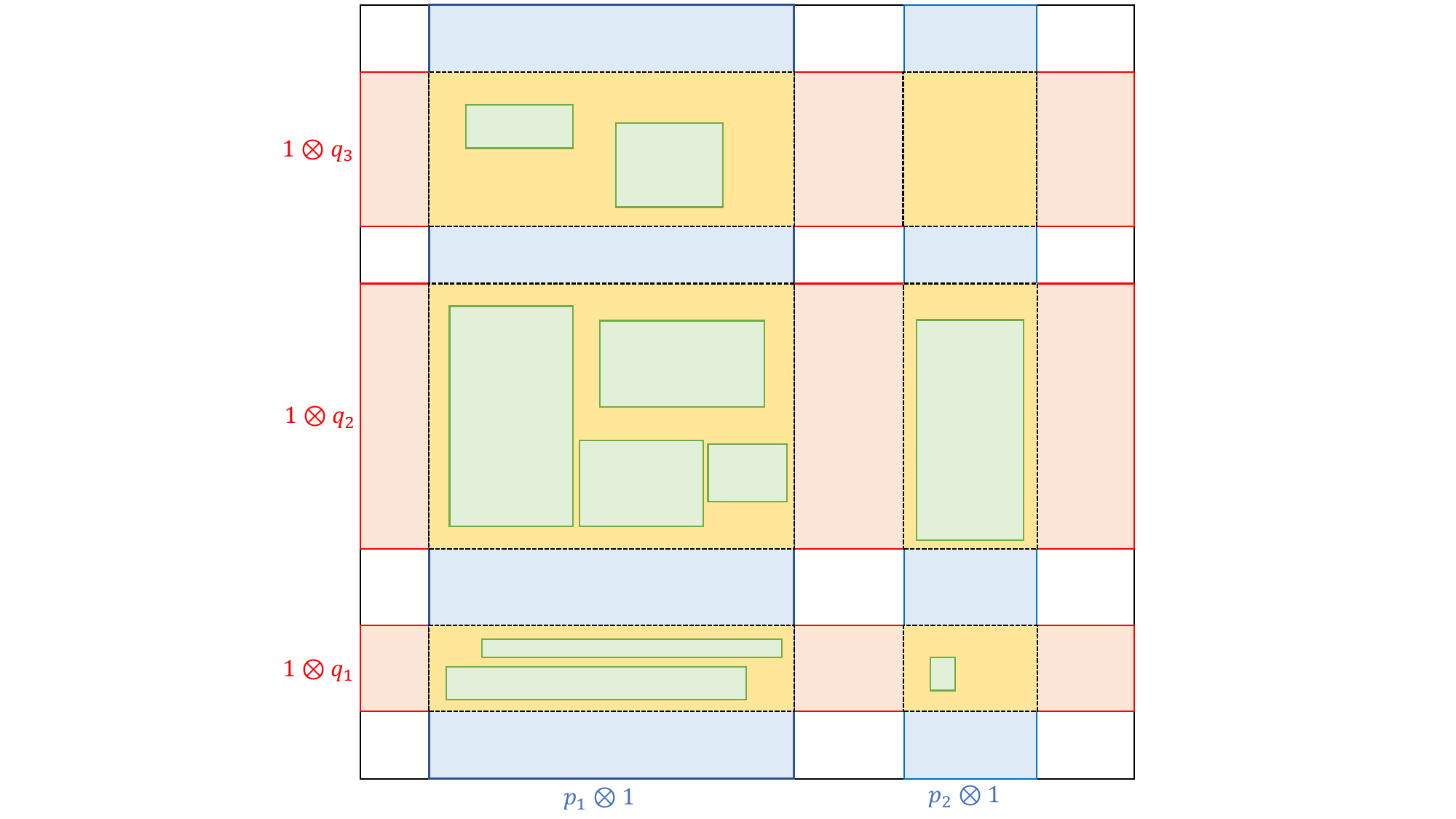}}
\end{figure}

\vspace{-1em}

Using the interchange relation in $C_n \otimes C_m$ we see that $c_1$ is of the form
\begin{align*}
     (p \otimes 1) \circ (t_1, \ldots, t_a) &= (p \otimes 1) \circ \left( (1 \otimes q) \circ (v_1^1, \ldots, v_b^1), \ldots, (1 \otimes q) \circ (v_1^a, \ldots, v_b^a) \right) \\
    &= \sigma_{a,b}^*  (1 \otimes q) \circ \left( (p \otimes 1) \circ (v_1^1, \ldots, v_1^a), \ldots, (p \otimes 1) \circ (v_b^1, \ldots, v_b^a ) \right)  \\
    &= (1 \otimes q) \circ (v_1, \ldots, v_b)
\end{align*}
where $\sigma_{a,b}$ is the permutation in $\Sigma_{ab}$ that appears in the interchange relation. In particular $v_i \neq \ast$ since otherwise the $i^{th}$ strip of $\mu'(c_1)$, as determined by the operation $\mu'(1 \otimes q)$, would have no cubes in its interior, which contradicts $s_i \neq \ast$. We are now in the position of applying case (ii) to conclude that $c_1=c_2$, which finishes the proof of injectivity.

\end{proof}

\begin{lemma}
Let $j \geq 2$. Then $\mu'\colon (C_n \otimes C_m)(j) \to (C_n \vert C_m)(j)$ is a proper map.
\label{proper}
\end{lemma}

\begin{proof}
Observe that if we have a surjective continuous map $\pi\colon \mathbf{F}(j) \rightarrow (C_n \otimes C_m)(j)$ then the properness of $\nu=\mu' \circ \pi$ implies the properness of $\mu'$ since then the equality $(\mu')^{-1}(K) = \pi \left( \nu^{-1}(K) \right)$ holds for any (compact) subset $K \subseteq (C_n \vert C_m)(j)$. In our case, we will consider the operad $\mathbf{F}$ obtained from the operad coproduct $C_n \amalg C_m$ by imposing the interchange relation between unary operations and operations of arbitrary arity, and we set $\mathbf{F}(j)$ to be its space of arity $j$ operations. This leads to the commutative diagram
$$\begin{tikzcd}
\mathbf{F}(j) \arrow[r, two heads, "\pi"] \arrow[rd, "\nu"'] & (C_n \otimes C_m)(j) \arrow[d, "\mu'~"] \\
                                                      & (C_n \vert C_m)(j),                     
\end{tikzcd}$$
where $\pi$ is the evident quotient map that imposes the remaining relations. 

Recall that a sequence $\{x_i : i \geq 1\}$ in a topological space
$X$ \textit{escapes to infinity} if, for every compact set $K \subseteq X$, the set $\{i \in \mathbb{N} : x_i \in K \}$ is finite. Moreover, for a continuous map between metric spaces $f \colon X \to Y$, the properness of $f$ is equivalent to showing that $f$ preserves all sequences that escape to infinity. With this in mind, we claim first that $\mathbf{F}(j)$ is a metric space: the operadic composition of $C_n$ and $C_m$ together with the specific form of the interchange relation we have imposed for $\mathbf{F}$ ensure that this space is a coproduct over the set $\mathbf{Tree}(j)$ of finite rooted trees with $j$ leaves and at most $j$ vertices with 2 incident edges, with each component being a product of spaces of the form $C_n(k)$ and $C_m(\ell)$, which are metric spaces. As for $(C_n \vert C_m)(j)$, this is a subspace of $C_{n+m}(j)$ and therefore is also metrizable.

 Let $\{ x_i : i \geq 1 \}$ be a sequence in $\mathbf{F}(j)$ that escapes to infinity. As $\mathbf{Tree}(j)$ is a \textit{finite} set, we can suppose that this sequence lies completely in one of the summands $T$ of the coproduct, and therefore consider instead the map
 $$ \nu_T : \prod_{s=1}^N C_{n_s}(a_s) \longrightarrow C_{n+m}(j)$$
 for certain integers $N$ and $a_s$, and $n_s$ being either $n$ or $m$. From this reduction, we see that there must exist some index $1 \leq s \leq N$ such that the projection of the sequence $\{ x_i \}$ to the component $C_{n_s}(a_s)$ defines a sequence that escapes to infinity. Since the sequences in $C_{n_s}(a_s)$ which escape to infinity are exactly the ones where the volume of one of the constituent cubes tends to 0, we conclude that one of the cubes in $\{ \nu_T(x_i) \}$ must have its volume also tending to 0, as required.
\end{proof}

 \begin{proof}[Proof of Theorem \ref{add}]
By diagram \eqref{main_diag}, it suffices to show that $\mu'\colon (C_n \otimes C_m)(j) \to (C_n \vert C_m)(j)$ is a homeomorphism for all arities $j \geq 0$. For $j=0$ this is clear since the operads in question have a single nullary operation. As for $j = 1$, the identification $(C_n \otimes C_m)(1)$ with $C_n(1) \times C_m(1)$ is due to the Boardman--Vogt product of categories coinciding with the cartesian product, and for $C_{n+m}(1)$ we use that any cube in $[0,1]^{n+m}$ is completely described by the cubes obtained by projecting to the first $n$ coordinates, and the last $m$ coordinates. Therefore, we may assume that $j \geq 2$.

From Lemma \ref{homeo} we get that $\mu'$ is a continuous bijection so we only have to check that $\mu'$ is closed. This follows from Lemma \ref{proper}, together with the fact that a proper continuous map $f\colon X \to Y$ with a locally compact and Hausdorff target is a closed map (with compact fibers).
\end{proof}

\section{Concluding remarks}

In this final section we make several final observations on a few aspects of the different proofs of Dunn additivity we previously mentioned, and on how one can extend our proof to accommodate for additivity of an arbitrary finite tensor product of little cubes operads. 

\begin{remark}
    We will begin with some comments on the existing literature concerning Dunn additivity.

\begin{enumerate}[leftmargin=0cm, itemindent =1cm]
    \item[(a)]  Dunn's definition of $D_k$ doesn't define an operad, since the concept of a decomposable little $k$-cube is not invariant under the action by the symmetric groups. However, this can be fixed by redefining a decomposable little cube to be a configuration of cubes $c \in C_k(j)$ such that there exists $\sigma \in \Sigma_j$ for which $\sigma^* c$ is a decomposable little cube in Dunn's sense (here $\sigma^* c$ denotes the action of $\sigma$ on $c$). Moreover, this small correction doesn't change any of the results in \cite{Dunn, Brinkmeier} since all the operads in question are $\Sigma$-free, that is, the symmetric groups act freely on the spaces of operations.
    \item[(b)] A fundamental point of the proof by Brinkmeier is not clear to us. More explicitly, to show that $\mu'$ admits a continuous inverse, he maps $C_n$ into an operad $\overline{C_n}$ where its spaces of operations, as well as the spaces of operations of $\overline{C_n} \otimes \overline{C_m}$, are compact and Hausdorff, and he further extends $\mu'$ to these new compactified versions of $C_n$. However, it is not clear to us how Brinkmeier's definition of $\overline{C_n}$ leads to an operad, nor how to then define $\overline{C_n} \otimes \overline{C_m}$.
    \item[(c)] The Boardman--Vogt tensor product of operads is easy to define, and is useful in describing mutually commuting algebraic structures on spaces. Its scope is limited, however, by the fact that it is not homotopy invariant. Lurie in \cite{LurieHA} presents a definition of a homotopy invariant analogue of the Boardman--Vogt tensor product in the context of $\infty$-operads, and proves a derived version of Dunn additivity for it. 
    \item[(d)] The results by Boavida de Brito and Weiss in \cite{BoavidaWeissProduct} on the configuration category of a product of manifolds provide a different proof for the additivity via the theory of complete Segal spaces. For more details, consult Section 6 of this reference.
    \item[(e)] By the monoidal equivalence proved in \cite{HinnichMoerdijk} between the Lurie model for $\infty$-operads and the dendroidal one, there is also an additivity result for the dendroidal version of the little cubes operads. It remains an open problem to provide a more direct proof of the latter additivity result.
    
\end{enumerate}
\label{history}
\end{remark}

\begin{remark}
One can consider the more general case of additivity for the tensor product of a finite number of operads of the form $C_n$. More precisely, in an analogous way to what we did at the beginning, we can construct an operad map
$$ \mu\colon C_{k_1} \otimes C_{k_2} \otimes \cdots \otimes C_{k_n} \longrightarrow C_{k_1+k_2+\cdots+k_n}$$
and ask whether $\mu$ is a  weak equivalence $C_{k_1} \otimes C_{k_2} \otimes \cdots \otimes C_{k_n} \simeq C_{k_1+k_2+\cdots+k_n}.$ As we have previously mentioned, this doesn't follow immediately from the case $n=2$ since the Boardman--Vogt tensor product in general won't preserve weak equivalences. In spite of this, we claim that the arguments explained above can be carried over to this generalized case where $n \geq 2$.

Indeed, one can upgrade the construction of the operads $C_n \vert C_m$ and $D_n \vert D_m$ to a multivariable version $C_{k_1} \vert \cdots \vert C_{k_n} $ and $D_{k_1} \vert \cdots \vert D_{k_n}$, see \cite{Brinkmeier}, to get a commutative diagram similar to \eqref{main_diag},
\begin{equation*}
\begin{tikzcd}
C_{k_1} \otimes \cdots \otimes C_{k_n} \arrow[rr, "\mu"] \arrow[d, "\mu'"'] &                                                                    & C_{k_1+\cdots+k_n}                            \\
C_{k_1} \vert \cdots \vert C_{k_n}                                               & D_{k_1} \vert \cdots \vert D_{k_n} \arrow[l, "\simeq"', hook'] \arrow[r, -, double equal sign distance] & D_{k_1+\cdots+k_n}. \arrow[u, "\simeq"', hook]
\end{tikzcd}
\end{equation*}

Therefore, it again suffices to show that $\mu'$ is an isomorphism. Lemma \ref{gen_pos} and Lemma \ref{proper} can be easily generalized to this new context. For the proof of Lemma \ref{homeo}, checking the injectivity of $\mu'$ involves studying an equation of the form $\mu'(c_1)=\mu'(c_2)$, where $c_1$ and $c_2$ are configurations which can be written as
$$c_1 = (1 \otimes \cdots \otimes p \otimes \cdots \otimes 1) \circ \mathbf{t} \hspace{1em} \mathrm{and} \hspace{1em} c_1 = (1 \otimes \cdots \otimes q \otimes \cdots \otimes 1) \circ \mathbf{s}$$
by the generalization of Lemma \ref{gen_pos}. Here $p$ and $q$ are, respectively, operations in the factors $C_{k_i}$ and $C_{k_j}$ of the tensor product, with $i$ and $j$ not necessarily equal. One can then proceed as before, by an argument only involving the $i^{th}$ and $j^{th}$ tensor factors.
\label{multi}
\end{remark}

\begin{remark}
Since the little 1-cubes operad $C_1$ coincides with its suboperad of decomposable cubes $D_1$, the isomorphism of operads 
$$\mu'\colon D_1^n = C_1^{\otimes n} \longrightarrow C_1 \vert \cdots \vert C_1 = D_n$$ 
that we obtained from the considerations in Remark \ref{multi} shows that the collection of decomposable little $n$-cubes operads $\left\{ D_n : n \geq 1 \right\}$ satisfies the stronger form of additivity $D_n \otimes D_m \cong D_{n+m}$. Notice that this statement also follows from Dunn's result (in contrast to the additivity of $C_n$) since the Boardman--Vogt tensor product will preserve isomorphisms of operads.
\end{remark}

\bibliographystyle{alpha}
\bibliography{bibliography.bib}
\end{document}